\documentclass{amsart}

\usepackage{amsmath}
\usepackage{amsfonts}
\usepackage{amssymb}

\newcommand{\N}{\mathbb{N}}

\newcommand{\R}{\mathbb{R}}
\newcommand{\C}{\mathbb{C}}
\newcommand{\T}{\mathbb{T}}
\newcommand{\D}{\mathbb{D}}

\newtheorem{theorem}{Theorem}
\newtheorem{proposition}[theorem]{Proposition}
\newtheorem{corollary}[theorem]{Corollary}

\theoremstyle{definition}
\newtheorem{remark}[theorem]{Remark}
\newtheorem{definition}[theorem]{Definition}

\newtheorem{conjecture}{Conjecture}

\DeclareMathOperator{\coker}{coker}

\begin{document}

\title{Essential positivity}

\author{A. Per\"al\"a}
\address{Department of Mathematics and Mathematical Statistics, Ume{\aa} University, 90187 Ume{\aa}, SWEDEN}
\email{antti.perala@umu.se}

\author{J. A. Virtanen}
\address{University of Reading, UK}
\email{j.a.virtanen@reading.ac.uk}



\begin{abstract}
We define essentially positive operators on Hilbert space as a class of self-adjoint operators whose essential spectra is contained in the nonnegative real numbers and describe their basic properties. Using Toeplitz operators and the Berezin transform, we further illustrate the notion of essential positivity in the Hardy space and the Bergman space.
\end{abstract}

\maketitle

\section{Essentially positive operators on Hilbert spaces}

Given a Hilbert space $H$ over the field of complex numbers $\C$, denote by $L(H)$ the algebra of bounded operators on $H$ and by $K(H)$ the two-sided ideal of compact operators in $L(H)$. The spectrum $\sigma(T)$ of an operator $T$ in $L(H)$ is defined by
$$
	\sigma(T)=\{\lambda \in \mathbb{C}:T-\lambda\ \text{\rm is not invertible}\}.
$$
Further, we write 
$$
	\ker T=\{x \in H:Tx=0\}\quad{\rm and}
	\quad \coker T=H/T(H)
$$ 
for the kernel and cokernel of $T$, respectively, and call an operator $T \in L(H)$ Fredholm and write $T\in \Phi(H)$ if both $\ker T$ and $\operatorname{coker}T$ are finite-dimensional. In this case, $T$ has a closed range and a well-defined Fredholm index:
$$
	\operatorname{ind}T=\operatorname{dim}\ker T-\operatorname{dim}\operatorname{coker}T.
$$
It is well known that the property of $T$ being a Fredholm operator is equivalent to invertibility of $[T]=T+K(H)$ in the Calkin algebra $L(H)/K(H)$. The essential spectrum $\sigma_{\rm ess}(T)$ of $T$ is defined by
$$
	\sigma_{\rm ess}(T)=\{\lambda \in \mathbb{C}:T-\lambda \notin \Phi(H)\}.
$$
This concept is sometimes called the Wolf (or Calkin or Fredholm) essential spectrum of $T$, whereas 
$$\bigcap_{K\in K(X)}\sigma(T+K)$$ is called the Weyl essential spectrum. Notice that $\sigma_{\rm ess} (T) = \sigma_{\rm ess} (T+K) \subset \sigma(T+K)$ for all $K\in K(H)$, so the Wolf essential spectrum is always contained in the Weyl essential spectrum. See \cite{FSW} for further details.

We say that a self-adjoint operator $T$ on the Hilbert space $H$ is positive and write $T\ge 0$ if
\begin{equation}\label{e:sa1}
	\sigma(T) \subset [0,\infty).
\end{equation}
Notice that there are operators that satisfy \eqref{e:sa1} but are not self-adjoint. Further, it is well known that $T\geq 0$ if and only if 
\begin{equation}\label{e:sa2}
	\langle Tx,x\rangle \geq 0\quad \text{\rm for all}\ x \in H
\end{equation}
which is often used as the definition of positive operators (see, e.g.,~\cite{Douglas}). Recall also that $T$ is a self-adjoint operator on $H$ if and only if $\langle Tx, x\rangle$ is real for all $x\in H$, and hence, in particular, every operator satisfying \eqref{e:sa2} is self-adjoint. We can now present the main concept of this paper.

\begin{definition}
A self-adjoint operator $T \in L(H)$ is said to be essentially positive if 
$$
	\sigma_{\rm ess}(T)\subset [0,\infty).
$$
In this case, we write $T \gtrsim 0$.
\end{definition}

Using the quotient algebra $L(H)/K(H)$, known as the \emph{Calkin algebra}, we can easily reduce the previous definition to positivity of elements in a $C^*$-algebra (i.e., of self-adjoint elements whose spectra are contained in $[0,\infty)$). More precisely, for an element $a$ in a $C^*$-algebra $\mathcal{A}$ with identity, we define the \emph{spectrum} of $a$ by $\sigma(a) = \{ \lambda \in \C : a-\lambda\ \textrm{is not invertible}\}$, which extends the definition of the spectrum of an operator above. Further, we say that an element $a\in \mathcal{A}$ is \emph{positive} if $a^*=a$ and $\sigma(a)\subset [0,\infty)$. For the fact that $L(H)/K(H)$ is a $C^*$-algebra, see Theorem 5.38 of \cite{Douglas}.

\begin{proposition}\label{calkin}
Let $H$ be a Hilbert space and $T\in L(H)$ be self-adjoint. Then $T$ is essentially positive if and only if $T+K(H)$ is positive in $L(H)/K(H)$.
\end{proposition}
\begin{proof}
It suffices to notice that $\sigma(T+K(H)) = \sigma_{\rm ess} (T)$.
\end{proof}

Recall that a net $(x_\alpha)$ in $H$ is said to converge weakly to $x\in H$ if
$$
	\lim_{\alpha} \langle x_\alpha,y\rangle = \langle x,y\rangle
$$
for all $y\in H$. Bounded operators preserve weak convergence, and compact operators turn it into strong convergence; that is, if $x_\alpha \to x$ weakly in $H$ and $K\in K(H)$, then $Kx_\alpha \to Kx$.

The following result gives a characterization of essentially positive operators on Hilbert spaces.

\begin{theorem}\label{main1}
Let $H$ be a Hilbert space and $T \in L(H)$ be self-adjoint. The following are equivalent:
\begin{itemize}
\item[(a)] $T$ is essentially positive;
\item[(b)] $T+K$ is positive for some $K\in K(H)$;
\item[(c)] $\liminf_{\alpha} \langle Tx_\alpha,x_\alpha\rangle \geq 0$ whenever $x_\alpha$ are unit vectors and $x_\alpha \to 0$ weakly.
\end{itemize}
\end{theorem}

\begin{proof} We assume that $T$ is self-adjoint, and prove $(a)\Rightarrow (b) \Rightarrow (c) \Rightarrow (a)$. 

Assume that (a) holds. By Proposition~\ref{calkin}, $T+K(H)$ is positive and hence there is some $S\in L(H)$ such that $T+K(H)=S^*S+K(H)$. Thus, there is a compact operator $K$ on $H$ such that
$$
	T + K = S^*S \ge 0,
$$
that is, (b) holds.

Assume next that (b) holds, and let $(x_\alpha)$ be a net of unit vectors with $x_\alpha\to 0$ weakly. Then, for some $K\in K(H)$, we have
$$
	\liminf_{\alpha} \langle Tx_\alpha,x_\alpha\rangle=\liminf_{\alpha} \left(\langle (T+K)x_\alpha,x_\alpha\rangle-\langle Kx_\alpha,x_\alpha\rangle\right)\geq 0,
$$
that is, (c) holds.

Finally, assume that (c) holds, and let $\lambda \in \sigma_{\rm ess}(T)$. Then $\lambda$ also belongs to the Weyl essential spectrum of $T$. Recall that the numerical range of $T$ is defined by 
$$
	W(T) = \{ \langle Tx,x \rangle : x\in H, \|x\|=1\}
$$
and the essential numerical range $W_{\rm ess}(T)$ can be described as
$$
	W_{\rm ess}(T)=\bigcap_{K \in K(H)}\overline{W(T+K)}
$$
(see, e.g., page 187 of \cite{FSW}). For every compact operator $K$, we have
$$
	\sigma_{\rm ess}(T)=\sigma_{\rm ess}(T+K)\subset \sigma(T+K)
	\subset \overline{W(T+K)}
$$
(see Exercise 4.4 of \cite{Douglas} for the last inclusion), and so $\sigma_{\rm ess}(T) \subset W_{\rm ess}(T)$. By Corollary to Theorem~5.1 of~\cite{FSW}, it follows that 
$$
	\lambda = \lim_{n\to \infty} \langle Tx_n,x_n\rangle
$$
for some sequence $(x_n)$ of unit vectors with $x_n\to 0$ weakly, which implies that $\lambda \in [0,\infty)$.
\end{proof}

We provide one more characterization of essential positivity for operators of the canonical form. 

\begin{proposition}\label{canonical}
Let $\{e_n\}$ be an orthonormal basis in $H$ and suppose that the operator
$$
	Tx = \sum_{n=1}^\infty \lambda_n\langle x,e_n\rangle e_n,\quad x\in H,
$$
is a bounded self-adjoint operator, so that each $\lambda_n\in\R$. Then $T$ is essentially positive if and only if $\liminf_{n\to\infty} \lambda_n \ge 0$.
\end{proposition}
\begin{proof}
Suppose first that $\liminf_{n\to\infty} \lambda_n \ge 0$. Let
$$
	\lambda_n^+ = \max \{\lambda_n, 0\} \quad{\rm and}\quad 
	\lambda_n^- = \lambda_n-\lambda_n^+,
$$
for all $n\in\N$, and define
$$
	T^{\pm} x = \sum_{n=1}^\infty \lambda^{\pm}_n\langle x,e_n\rangle e_n
$$
for $x\in H$. Then $T^+$ is positive because $\lambda^+_n\ge 0$ for all $n$, and $T^-$ is compact because $\lambda^-_n\to 0$. Therefore, $T = T^+ + T^-$ is essentially positive by Theorem~\ref{main1}.

Conversely, suppose that $T+K$ is positive for some compact operator $K$. Notice first that $e_n \to 0$ weakly because $\{e_n\}$ is orthonormal, and hence $Ke_n\to 0$. Therefore, since
$$
	\langle (A+K)e_n,e_n\rangle 
	= \lambda_n + \langle Ke_n,e_n\rangle\ge 0,
$$
we have $\lambda_n \ge - \langle Ke_n, e_n\rangle \to 0$ as $n\to \infty$, which shows that $\liminf_{n\to \infty} \lambda_n \ge 0$.
\end{proof}

In the subsequent sections, we will illustrate essential positivity in the context of the Berezin transform and Toeplitz operators on the Hardy, Bergman and Fock spaces.

\section{The Berezin transform}

Let $\Omega$ be a set and suppose that $H$ is a reproducing kernel Hilbert space on $\Omega$. That is, for each $z\in \Omega$, the point evaluation $E_z : H\to \C$ defined by $E_z(f) = f(z)$ is bounded, and hence there is a unique vector $K_z\in H$ such that for each $f\in H$,
$$
	f(z) = \langle f, K_z\rangle.
$$
We call $K_z$ the reproducing kernel for $H$. The normalized reproducing kernel $k_z$ is defined by $k_z = K_z/\|K_z\|$. Given $T\in L(H)$, the Berezin transform $\tilde T$ is defined by
$$
	\tilde T(z) = \langle Tk_z, k_z\rangle
$$
for $z\in \Omega$. The Berezin transform can be used to obtain information about concrete operators acting on various (analytic) function spaces.

We note a couple of simple but useful properties. If $T$ is positive, then $\tilde T(z) \ge 0$ for all $z\in \Omega$. The following result is a consequence of the Cauchy-Schwarz inequality.

\begin{proposition}
Let $T$ be compact on $H$ and $\Omega\subset \C$ be open. If $k_z\to 0$ weakly as $z\to \partial \Omega$, then $\tilde T(z)\to 0$ as $z\to \partial \Omega$. 
\end{proposition}

\begin{proposition}\label{Berezin1}
Let $T$ be essentially positive on $H$ and $\Omega\subset \C$ be open. If $k_z\to 0$ weakly as $z\to\partial\Omega$, then
$$
	\liminf_{z\to \partial \Omega} \tilde T(z) \geq 0.
$$
\end{proposition}
\begin{proof}
If $T$ is essentially positive, then there is a compact operator $K$ such that $T+K\ge 0$. Thus,
$$
	\liminf_{z\to\partial\Omega}\tilde T(z) 
	\ge \liminf_{z\to \partial\Omega} (\tilde T+ \tilde K)(z) - \lim_{z\to\partial\Omega} |\tilde K(z)| \ge 0
$$
by the previous proposition.
\end{proof}
In Section~\ref{Bergman}, we will see that the converse of the preceding proposition is not true in general for operators acting on the Bergman space.

\section{Toeplitz operators on the Hardy space}\label{Hardy}

The Hardy space $H^2$ of the unit circle $\T$ is defined by
$$
	H^2 = \{ f\in L^2(\T) : f_k = 0\ {\rm for}\ k<0\},
$$
where $f_k$ denotes the $k$th Fourier coefficient of $f$. Denote by $P$ the orthogonal projection of $L^2(\T)$ onto $H^2$, and define the Toeplitz operator $T_f$ with symbol $f$ by
$$
	T_f g = P(fg)
$$
for $g\in H^2$. A theorem of Brown and Halmos shows that $T_f$ is bounded on $H^2$ if and only if $f\in L^\infty$ (see, e.g.,~\cite{BS2006}). Notice that the adjoint $T^*_f$ is $T_{\bar f}$, so $T_f$ is self-adjoint if and only if $f$ is real almost everywhere.

Using well-known spectral results, it is easy to characterize essential positivity of Toeplitz operators on $H^2$.

\begin{proposition}\label{H2}
Suppose that $f\in L^\infty(\T)$ is real-valued. Then the following conditions are equivalent:
\begin{enumerate}
\item $T_f$ is positive on $H^2$;

\item $T_f$ is essentially positive on $H^2$;

\item $f(t) \in [0,\infty)$ for almost every $t\in\T$.
\end{enumerate}
\end{proposition}
\begin{proof}
Clearly (i) implies (ii). If $f(t)\in [0,\infty)$ for a.e.~$t\in\T$, then, according to a theorem of Brown and Halmos (see, e.g., Theorem~2.33 of~\cite{BS2006}), the spectrum $\sigma(T_f)$ is contained in the convex hull of the essential range of $f$, so $\sigma(T_f)\subset [0,\infty)$. We see that (iii) implies (i).

Finally, if $T_f$ is essentially positive, then, by a theorem of Hartman and Wintner (see, e.g., Theorem 2.30 of \cite{BS2006}), we have $f(t) \in \sigma_{\rm ess}(T_f) \subset [0,\infty)$ for a.e.~$t\in\T$. So (ii) implies (iii).
\end{proof}

\section{Toeplitz operators on the Bergman space}\label{Bergman}

Denote by $H(\D)$ the space of all analytic functions in the unit disk $\D$. The Bergman space $A^2$ is defined by
$$
	A^2 = \left\{ f\in H(\D) : \int_{\D} |f(z)|^2 \, dA(z)<\infty \right\},
$$
where $dA(z) = \frac{dx dy}{\pi}$ with $z=x+iy$ is the normalized area measure. 

Assume now that $\mu$ is a complex-valued Borel measure on $\mathbb{D}$. Then the Toeplitz operator $T_\mu : A^2 \to A^2$ is given by
\begin{equation}\label{e:integral rep}
	T_\mu(g)(z) = \int_{\D} g(w) \overline{K_z(w)}\, d\mu(w) 
	= \int_{\D} \frac{g(w)\, d\mu(w)}{(1-\bar w z)^2},
	\quad z\in \D,
\end{equation}
where $K_z(w) = (1-\bar z w)^{-2}$ is the reproducing kernel for $A^2$. Notice that $T_\mu(g)$ is well-defined for all polynomials $g$ (complex-valued measures $\mu$ are assumed to have finite total variation $|\mu|$). It is easy to see that $T_\mu$ is self-adjoint if and only if $\mu$ is real-valued and that $T_\mu$ is positive if $\mu$ is non-negative. The converse of the latter statement fails---see~\cite{ZZ2014}. As usual, if $\mu$ is given in terms of an $L^1$ function $f$, that is, $d\mu(z)=f(z)dA(z)$, then we write $T_f$ for the corresponding Toeplitz operator.

The Berezin transform $\tilde \mu$ is defined by
$$
	\tilde \mu(z) = \int_{\D} |k_z(w)|^2 d\mu(w), \quad z\in \D.
$$
If $T_\mu$ is bounded on $A^2_\alpha$, then the Berezin transform of $T_\mu$ is given by
\begin{equation}\label{e:Berezin-Toeplitz}
	\widetilde T_\mu(z) = \tilde \mu (z)
\end{equation}
for $z\in \D$ (see page 165 of \cite{Zhu}).

It may be tempting to try to relate positivity of Toeplitz operators $T_\mu$ on the Bergman space to positivity of the Berezin transform of $\mu$, but this is known to fail in general. More precisely, there are functions $f$ such that $\tilde f \ge 0$ but $T_f$ is not positive, while, however, for harmonic symbols $f$, $T_f \ge 0$ if and only if $\tilde f\ge 0$ (see \cite{ZZ2014}). It would be of interest to determine a larger class of functions for which the preceding characterization holds.

We next consider essential positivity. Recall first that, by Proposition~\ref{Berezin1}, 
$$
	T  \gtrsim 0 \implies \liminf_{|z|\to 1} \tilde T(z)\ge 0.
$$
The converse is not true in general; for example, if
$$
	T\left( \sum_{n=0}^\infty a_n z^n \right) = - \sum_{n=0}^\infty a_{2^n} z^{2^n},
$$
then $\tilde T(z) \to 0$ as $|z|\to 1$ (see \cite{AZ1998}), but for any compact operator $K$,
$$
	\langle (T+K) e_{2^n}, e_{2^n}\rangle = -1 + \langle Ke_{2^n}, e_{2^n}\rangle \to -1
$$
where $e_n = z^n/\|z^n\|$. However, we can characterize essentially positive Toeplitz operators with real-valued radial symbols under a mild additional condition. A measure $\mu$ is said to be radial if $\mu(E)=\mu(e^{i\theta}E)$ for all $\theta \in \mathbb{R}$ and Borel sets $E$.

We first recall the following Tauberian theorem (see Theorem I.11.1 of \cite{Korevaar}). Tauberian theorems were used for Toeplitz operators by Korenblum and Zhu in \cite{KZ1995}. Inspired by their work, we will use similar approach to prove the main result of this section. 

\begin{theorem}\label{Tauberian}
Suppose that $\sum_{n=1}^\infty a_n x^n$ converges for $|x|<1$ and let
$$
	f(x) = (1-x) \sum_{n=0}^\infty a_n x^n.
$$
If $f(x) \to a$ as $x\to 1$, then the Tauberian condition $a_n\ge -C$ for some constant $C$ implies that
$$
	\frac{a_0+a_1 + \ldots a_n}{n+1} \to a.
$$
\end{theorem} 

Observe that for a radial measure $\mu$ as discussed above the Toeplitz operator $T_\mu$ is diagonal with respect to the orthogonal monomial basis $(e_n)_{n=0}^\infty$, where $e_n(z)=\sqrt{n+1}z^n$, and the diagonal elements are given by
\begin{equation}\label{e:d}
	\lambda_n = 2(n+1)\int_0^1 r^{2n+1}\,d\mu(r).
\end{equation}
It is clear that the operator $T_\mu$ is bounded if and only if the numbers $\lambda_n$ form a bounded sequence. However, we shall assume something slightly stronger, namely that $|\mu|$ is a Carleson measure. 

Carleson measures have been characterized for all positive Borel measures, but in the case of a radial measure the condition can be simplified. In what follows, we will make use of the standard weighted Bergman spaces $A^2_\alpha$ (so that $A^2=A^2_0$), consisting of analytic functions, which are square-integrable with respect to the weighted measure $(1-|z|^2)^\alpha\, dA$, where $\alpha>-1$. A positive radial Borel measure $\nu$ is a Carleson measure for $A^2_\alpha$, meaning that it satisfies
$$\int_{\mathbb{D}}|P(z)|^2 d\nu(z)\lesssim \int_{\mathbb{D}}|P(z)|^2 (1-|z|^2)^\alpha dA(z)$$ for all polynomials $P$, if and only if
$$\sup_{0\leq r <1} \frac{1}{(1-r^2)^{2+\alpha}}\int_r^1 d\nu(s)<\infty.$$
This result is standard and can be found in \cite{Zhu}, for instance. In that reference it is formulated in terms of disks in Bergman metric---the relation with the formulas above is more evident from the characterization with Carleson boxes; see \cite{PRMem}.

Observe that if $|\mu|$ is a Carleson measure for $A^2$, then it holds that
$$\frac{1}{(1-r^2)^3}\int_r^1 (1-s^2)d|\mu|(s)\leq \frac{1}{(1-r^2)^2}\int_r^1 d|\mu|(s).$$
Indeed, just notice that with $s$ on the interval $[r,1)$ it always holds $(1-s^2)\leq (1-r^2)$. This reasoning shows that $(1-|z|^2)|\mu|$ is a Carleson measure for $A^2_1$.

\begin{theorem}\label{char1}
Let $\mu$ be a radial real-valued Borel measure such that $|\mu|$ is a Carleson measure for $A^2$. Suppose that the limit $L=\lim_{|z|\to 1} \tilde \mu(z)$ exists. Then $T_\mu$ is essentially positive on $A^2$ if and only if $L\ge 0$.
\end{theorem}
\begin{proof}
As above, we write
$$
	T_\mu(g) = \sum_{n=0}^\infty \lambda_n\langle g, e_n\rangle e_n,
$$
where $\lambda_n$ are given by \eqref{e:d}. A direct computation shows that
\begin{align*}
	\tilde\mu(z) &= 2(1-|z|^2)^2 \sum_{n=0}^\infty (n+1)\lambda_n |z|^{2n}\\
	&= 2(1-|z|^2) 
	\left( \lambda_0 + \sum_{n=1}^\infty \big((n+1)\lambda_n-n\lambda_{n-1}\big)|z|^{2n}\right).
\end{align*}
The sequence of numbers $(n+1)\lambda_n - n\lambda_{n-1}$ with $n\ge 1$ is bounded. Indeed,
\begin{align*}
(n+1)\lambda_n - n\lambda_{n-1}&=-n^2\int_{\mathbb{D}}|z|^{2n-2}(1-|z|^2)d|\mu|(z)\\
&\qquad+(2n+1)\int_{\mathbb{D}}|z|^{2n}d|\mu|(z).
\end{align*}
The first term of the sum is bounded in $n$, because $(1-|z|^2)|\mu|$ is a Carleson measure for $A^2_1$, for then
\begin{align*}
n^2\int_{\mathbb{D}}|z|^{2n-2}(1-|z|^2)d|\mu|(z)&\lesssim n^2\int_{\mathbb{D}}|z|^{2n-2}(1-|z|^2)dA(z)\\
&=2n^2\int_0^1 r^{2n-1}(1-r^2)dr\\
&=2n^2\left(\frac{1}{2n}-\frac{1}{2n+2}\right)\\
&=\frac{4n^2}{4n^2+4n}.
\end{align*}

The boundedness of the second term follows from the assumption that $|\mu|$ is a Carleson measure for $A^2$ by a similar but easier calculation.

Observe now that if $a_0=\lambda_0$ and $a_n=(n+1)\lambda_n-n\lambda_{n-1}$ for $n\geq 1$, then we have
$$\frac{a_0+a_1+...+a_n}{n+1}=\lambda_n$$ for $n\ge 0$.

By Theorem~\ref{Tauberian}, 
$$
	\lim_{t\to 1} (1-t)^2 \sum_{n=0}^\infty (n+1)\lambda_n t^n = L
$$
if and only if $\lambda_n \to L$. Now Proposition~\ref{canonical} shows that the condition $L\ge 0$ is sufficient for $T_\mu$ to be essentially positive. That the condition is also necessary follows from \eqref{e:Berezin-Toeplitz} and Proposition~\ref{Berezin1}.
\end{proof}

The following corollary follows immediately from the preceding theorem. 

\begin{corollary}\label{cor1}
Let $f\in L^\infty(\D)$ be real and radial, and suppose that the limit $L=\lim_{|z|\to 1} \tilde f(z)$ exists. Then $T_f$ is essentially positive on $A^2$ if and only if $L\ge 0$.
\end{corollary}

\begin{remark}
We note that alternatively Corollary \ref{cor1} follows easily from Theorem 2.2 of Axler and Zheng~\cite{AZ1998}, which states that, for $f\in L^\infty(\D)$, the Toeplitz operator $T_f$ is compact if and only if $\bar f(z) \to 0$ as $|z|\to 1$. This theorem has also been extended by Su\'arez \cite{Sua} to characterize compactness of all operators in the Toeplitz algebra. 

After the first version of this paper appeared on arXiv, we were contacted by Daniel Su\'arez and Nina Zorboska, who drew our attention to their articles \cite{SuaCrelle} and \cite{Zor}. Both of these papers imply that for $\mu$ whose total variation is a Carleson measure, compactness of $T_\mu$ is characterized by the vanishing of the Berezin transform on the boundary. This leads to an easier proof of Theorem \ref{char1} for all such measures whose Berezin transforms have continuous boundary values. The authors want to express their gratitude to Professors Su\'arez and Zorboska for bringing these results to their attention.
\end{remark}

The virtue of the proof of Theorem~\ref{char1} is that it suggests the following conjecture, which is in fact similar in spirit to a compactness characterization for such symbols in \cite{KZ1995}.

\begin{conjecture}\label{conjecture}
Let $f\in L^\infty(\D)$ be real valued and radial. Then $T_f : A^2 \to A^2$ is essentially positive if and only if
\begin{equation}\label{e:conjecture}
	\liminf_{|z|\to 1}\tilde f(z) \ge 0.
\end{equation}
\end{conjecture}

In \cite{ZZ2014} Zhao and Zheng gave examples of symbols of the form $f(z)=|z|^2+a|z|+b$, for which the Berezin transform is positive, while the corresponding Toeplitz operator is not. These examples still obey the conjecture above. In fact, it is immediate that the conjecture is satisfied by all symbols in either the Douglas algebra $C(\overline{\D})+H^\infty(\D)$ or in $L^\infty(\D)\cap \mathrm{VMO}(\D)$, because in these cases the essential spectra can be described in terms of the Berezin transform, see \cite{McDonald, ZhuVMO}. This suggest that Conjecture~\ref{conjecture} may hold true also for non-radial symbols.

\begin{remark}
Naturally a similar theory can be developed for Toeplitz operators $T_f$ acting on the Fock space of the complex plane and it is indeed easy to extend Corollary~\ref{cor1} to this setting. However, we are not aware of a result, such as Theorem~\ref{Tauberian}, in the planar setting that could be used to prove the Fock space version of Theorem~\ref{char1}. We still assume that Conjecture 1 with obvious modifications is true also in the Fock space setting.
\end{remark}

\bibliographystyle{amsalpha}

\end{document}